\documentclass{article}
\usepackage{graphicx}
\usepackage{amsmath}
\usepackage{amssymb}
\usepackage{amsthm}
\usepackage{soul}
\usepackage{biblatex}

\usepackage[a4paper, total={6in, 8in}]{geometry}

\usepackage{verbatim}

\DeclareMathOperator{\diam}{diam}

\DeclareMathOperator{\UW}{UW}
\DeclareMathOperator{\W}{W}
\DeclareMathOperator{\Int}{Int}
\DeclareMathOperator{\Hom}{Hom}

\DeclareMathOperator{\e}{\varepsilon}

\DeclareMathOperator{\Z}{\mathbb{Z}}
\DeclareMathOperator{\N}{\mathbb{N}}
\DeclareMathOperator{\R}{\mathbb{R}}

\theoremstyle{definition}
\newtheorem{definition}{Definition}

\newtheorem{lem}{Lemma}

\newtheorem{rem}{Remark}
\newtheorem{thm}{Theorem}

\title{Urysohn Width and Surgeries}
\author{Aleksandr Berdnikov, Brendan Isley}

\begin{document}
\maketitle
\begin{abstract}
    We analyze the behavior of Urysohn width of manifolds under a connected sum operation, specifically, bounding widths of summands in terms of widths of the sum and vice versa. Our methods also apply to the universal covers of these spaces, and to more general type of surgeries. Lastly, we provide examples that show the optimality of constants in our estimates.
\end{abstract}
\section{Introduction}
Urysohn width is a classical notion in dimension theory, initially posed in the 1920's, which is used to tell us how well we can approximate a space by a $k$-dimensional space for any $k$. Gromov reignited the subject in the 1980's and 1990's within the context of scalar curvature and related notions \cite{gromov1996positive, guth2017volumes}, as well as systolic geometry \cite{gromov1983filling, guth2017volumes, liokumovich2023waist}. A comprehensive history of Urysohn width can be found in the thesis \cite{balitskiy2021bounds}, Section 2.

We state the relevant definitions. 
\begin{definition}
The \emph{width} $\W(f)$ of a mapping $f:X\to Y$ is the size of its largest fiber:
$$\W(f):=\sup_{y\in Y}\diam(f^{-1}(y))$$
\end{definition}

\begin{definition}\label{def_complete}
The \emph{Urysohn $k$-width} $\UW_k(X)$ of a complete metric space $X$ is the smallest width of its projection to a $k$-dimensional space: 
$$\UW_k(X):=\inf\{w\text{ for which }\exists P^k, f:X\to P^k\text{ such that }\W(f)=w\},$$
where $P^k$ is a simplicial $k$-complex and $f$ is a proper map.  
\end{definition} 

We will often call it merely the \emph{$k$-width}. Intuitively, the $k$-width measures how well the space $X$ can be approximated by a $k$-dimensional object. Notice that the $k$-width is bounded above by the $(k-1)$-width, so lower-dimensional width bounds are better when possible, but are often more difficult to prove. 

The \emph{macroscopic dimension} of a space $X$ is defined as 
$$
\dim_{mc}(X) = \min\{k \geq 0\,:\, \UW_k(X) < \infty\}.
$$
For instance, while the cylinder $S^1 \times \R$ has topological dimension equal to $2$, its macroscopic dimension is equal to $1$.

Understanding the relationship between Urysohn width and surgeries can be a tool for proving useful results. For instance, the Kneser-Milnor prime decomposition theorem \cite{milnor1962unique} states that a compact $3$-manifold $M$ can be decomposed uniquely into prime $3$-manifolds under connect sum:
$$
M = X_1 \# \dots \# X_s \# (S^2 \times S^1) \# \dots \# (S^2 \times S^1) \# K_1\#\dots\#K_t,
$$
where the fundamental groups of each $X_i$ are finite and each $K_j$ is aspherical. An important result of Gromov in the theory of scalar curvature states that a compact $3$-manifold $M$ admitting a metric with positive scalar curvature contains no aspherical connected summand in its Kneser-Milnor prime decomposition (\cite{gromov2019four}, Section 3.10. The reverse implication is also true). It is known that any complete $3$-manifold with positive scalar curvature, in particular the universal cover of a compact $3$-manifold with positive scalar curvature, has  bounded $1$-width (\cite{gromov1983positive,gromov2019four}). Gromov's proof then involves using this to prove a restriction on $\pi_1(M)$ using Stalling's theorem from geometric group theory. On the other hand, if we know a bound on the $1$-width of the universal cover of a connect sum implies a bound on the $1$-width of the universal cover of each summand, then we have an alternative proof of this result as a direct consequence of another well-known result: Gromov's result that the macroscopic dimension of the universal cover of an aspherical $n$-manifold is equal to $n$ (\cite{gromov1996positive}, Section 2$\frac{1}{2}$).

Another possible relation between Urysohn width and surgeries comes from a question which is natural and interesting in its own right: If we have two manifolds with bounded $k$-width, would a manifold constructed by joining these manifolds through surgery  also have bounded $k$-width?

In this paper, we address these questions, and give answers in quantitative terms for some particular cases. To formulate our results,  let us describe the type of surgeries we consider. Let $M_1$ and $M_2$ be two $n$-dimensional manifolds, and suppose $A$ is a compact $n$-dimensional manifold with boundary which is embedded into both $M_i$ (these embeddings are presumed to preserve the Riemannian metric of $A$ throughout this paper). By $M_1 \#_A M_2$, we mean the space defined by the following connected sum type construction: we remove the interior of the embedded $A$ in both $M_1$ and $M_2$ and identify $\partial A$ in $M_1$ and $M_2$ (the result may not be a smooth manifold along $\partial A$). In particular, when $A$ is the ball $B^n\subset \Int(M_i)$, then topologically this coincides with the connected sum of $M_1$ and $M_2$.

We are now ready to state our main theorems. 
\begin{thm}\label{nocover}
    Let $M_1$ and $M_2$ be complete $n$-manifolds, each containing an embedded copy of a compact $n$-manifold $A$ with boundary. Then for $i\in \{1,2\}$ and $k\in \N$ we have
    $$\UW_{k}(M_i)\leq 2\UW_{k}(M_1 \#_A M_2)+\diam_{M_i}(A)$$
    In particular, if $M_1 \#_A M_2$ has macroscopic dimension bounded above by $k$, then so does every $M_i$.
\end{thm}

{\bf Theorem~\ref{nocover}'.} If $M_1$, $M_2$ are complete manifolds, each containing an embedded copy of a compact n-manifold $A$ with boundary, then for $k\in\N$ we have
    $$\UW_k(M_1\#_A M_2)\leqslant \UW_k(M_1)+\UW_k(M_2)+\diam_{M_1\# M_2}(\partial A).$$
We can lower the upper bound in these inequalities under certain assumptions:
\begin{thm}\label{nocover1}
    If in the setting of theorem~\ref{nocover} we have $\dim(M_i)\geqslant 3$, and either
    \begin{itemize}
        \item[a)] $k=1$, and $\partial A$ is simply connected \newline (or, more generally, $\pi_1(\partial A)\to0\in \pi_1(M_i\setminus \Int(A))$ for some $i$, or $H^1(\partial A,\Z)\cong0$), or
        \item[b)] $k=n-1$ while $M_i$ are complete oriented without boundary and $\partial A$ is connected,
    \end{itemize}
    then we have a better estimate: 
    $$\UW_{k}(M_i)\leq \UW_{k}(M_1 \#_A M_2)+\diam_{M_i}(A).$$
\end{thm}
{\bf Theorem~\ref{nocover1}'.} If the conditions a) and b) of theorem~\ref{nocover1} are altered to
\begin{itemize}
    \item [a)] $k=1$ and $\pi_1(\partial A)\to0\in \pi_1(M_i)$ for each $i$ (or $H^1(\partial A, \Z)\cong0$), or
    \item [b)] $k=n-1$, $\partial A$ is connected, and either $A$ or both $M_i\setminus \Int(A)$ are complete oriented with only boundary being $\partial A$ that is connected, or 
    \item [c)] $A$ is a ball $B^n$, then
\end{itemize}
$$\UW_k(M_1\#_A M_2)\leqslant \max_{i}(\UW_k(M_i))+\diam_{M_1\# M_2}(\partial A).$$
The methods of theorem~\ref{nocover1}, unlike theorem~\ref{nocover}, do not compound the terms in the estimate when applied to several connected sums at once, that is, when a single $M$ has multiple $A_i$ that join it to other separate $M_i$. We use this feature in the particular case of universal covers. In this case, we are able to show that certain macroscopic dimension bounds on the universal cover of a connect sum imply the same bound on the universal cover of each summand. However, we do not have a universal constant in the Urysohn width estimates, unlike the previously stated results. 

Denote the universal cover of $X$ by $\tilde{X}$; then

\begin{thm}\label{cover}
    Let $M$ and $M'$ be complete $n$-manifolds ($n\geqslant3$) each containing an embedded copy of a connected, compact $n$-manifold $A$ with boundary. Suppose $A$ and $\partial A$ are simply-connected and the distance function on $\partial A$ is at most $C(\geqslant 1)$ times larger in $\tilde{M}$ than what it is in $\widetilde{M\#_AM'}$. Then
    $$\UW_{1}(\widetilde{M})\leqslant C\UW_{1}(\widetilde{M \#_A M'})+2\diam_{\tilde{M}}(A).$$
    If, in addition, $M$ and $M'$ are without boundary, then the same estimate holds for $n-1$ width:
    $$\UW_{n-1}(\widetilde{M})\leqslant C\UW_{n-1}(\widetilde{M \#_A M'})+2\diam_{\tilde{M}}(A).$$
\end{thm}

In particular, if $\widetilde{M \#_A M'}$ has macroscopic dimension bounded above by $1$ or by $n-1$, then the same is true for $\widetilde{M}$.

Let us comment on the constant $C$ in the theorem. The need for it arises from the fact that although we assume that the Riemannian metric on $\partial A$ is the same in every space it sits it, that may not be the case for the distance function, since there might be different shortcuts through different ambient spaces. If $\widetilde{M\setminus A}$ is geodesically convex in $\widetilde{M}$, then there are no shortcuts, implying $C=1$. In general, 
$$C=\sup_{x,y\in\partial A}\frac{d_{\tilde{M}}(x,y)}{d_{\widetilde{M\#_AM'}}(x,y)},$$ 
where $x$ and $y$ are meant to reside in a single copy of $\partial A$ in universal covers. Note that the compactness and smoothness of $\partial A$ guarantees $C<\infty$, and that $C\geqslant 1$, since on infinitesimal scale the distance function converges to that of $\partial A$. 

Another way around this technicality is to bound the number of copies of $\partial A$'s in $\widetilde{M\#_AM'}$ that may be intersected by a single fiber of an optimal projection $f:\widetilde{M\#_AM'}\to X^{n-1}$. For example, if the distance between different copies of $\partial A$ in $\widetilde{M\#_AM'}$ is more than $\UW_{n-1}(\widetilde{M\#_AM'})$, then the situation is essentially reduced to the settings of theorem~\ref{nocover1} and we get analogous estimates for the covers (this is addressed in remark~\ref{sparse}).

Now the reverse inequality:

{\bf Theorem~\ref{cover}'.} 
Let $M$ and $M'$ be complete $n$-manifolds each containing an embedded copy of a simply-connected, compact $n$-manifold $A$ with simply-connected boundary $\partial A$. Suppose the distance function on $\partial A$ is at most $C'(\geqslant 1)$ times larger in $\widetilde{M\#_AM'}$ than what it is in $\widetilde{M}$. Then
$$\UW_k(\widetilde{M\#_AM'})\leqslant C'\max_i(\UW_k (\widetilde{M_i}))+2\diam_{\widetilde{M\#_AM'}}(\partial A)$$
as long as one of the following conditions hold:
\begin{itemize}
    \item[a)] $k=1$, or
    \item[b)] $k=n-1$ and $M$ and $M'$ are without boundary, or
    \item[c)] $A$ is a ball $B_n.$
\end{itemize}

Note that $C'$ in this theorem is different from $C$ in the original, as it compares metrics on $\partial A$ in the opposite direction. That is,
$$C'=\sup_{x,y\in\partial A}\frac{d_{\widetilde{M_1\#_AM_2}}(x,y)}{d_{\widetilde{M_i}}(x,y)},$$
where $x$ and $y$ are meant to reside in a single copy of $\partial A$ in universal covers. 
 
We note that a non-quantitative version of Theorem \ref{cover} for $n-1$-width was proven for connect sums in (\cite{di2025curvature}, Theorem 7.14) using a different technique. 

The paper is organized as follows. In Section~\ref{nocover_section} we provide the proof of Theorem \ref{nocover}. In Section~\ref{contract_section}, we prove Theorem \ref{nocover1}, and in Section~\ref{cover_section'}, we show how to adapt the proof of Theorem \ref{nocover1} to prove Theorem \ref{cover}. In Section~\ref{counter_section} we list counterexamples to eliminating various unfortunate constants in our estimates.

The authors would like to thank Yevgeny Liokumovich for many helpful conversation regarding this paper.

\section{Simple ``No Covers'' Case: collapsing $A$}
\label{nocover_section}
In this section, we prove Theorem \ref{nocover} using a straightforward approach. We first need a lemma which allows us to assume that the maps which we use to measure Urysohn width are simpicial maps. This is taken from (\cite{balitskiy2022waist}, Lemma 3.9).
\begin{lem}\label{simp}
    Suppose $X$ is a topological space with a metric and a simplicial structure\footnote{that is, the metric doesn't have to be flat on simplices} and $h: X \to K$ is a proper map into a simplicial complex $K$ with $\diam(h^{-1}(z))<W$ for all $z \in K$. Then there is a simplicial proper map $h': X' \to K'$ such that $\diam((h')^{-1}(z))<W$ for all $z \in K'$ (where $X'$ and $K'$ are subdivisions of $X$ and $K$).
\end{lem}

\begin{proof}
    Let $K'$ be a subdivision of $K$ such that the preimage of the open star $S_v$ of every vertex $v \in K'$ has diameter less than $W$. Use the simplicial approximation theorem to approximate $h$ by a simplicial (for some subdivision of $X$) map $h'$ such that for each $x\in X'$, $h'(x)$ is contained in the minimal closed cell containing $h(x)$. It follows that each fiber of $h'$ is contained in the preimage of some open star $S_v$ in $K'$, and hence has diameter less than $W$.
\end{proof}

Now we prove Theorem \ref{nocover}.
\begin{proof}[Proof of theorem~\ref{nocover}]
    Suppose $f: M_1 \#_A M_2 \to P^k$ is a proper map with finite width $W$, which we pick arbitrarily close to $\UW_k(M_1 \#_A M_2)$. We want to extend this map from $M_1\setminus A$ to the whole $M_1$ by sending all of $A$ to a single point. Since $f$ is already defined on $\partial A$, its image $f(\partial A)$ needs to be collapsed in $P^k$. We'd like the resulting space $P^k/f(\partial A)$ to still be a simplicial $k$-complex; the rest of the paragraph will deal with this technicality. By lemma~\ref{simp} we may assume that $f$ is simplicial for some triangulation of the pair $(M_1 \#_A M_2,\partial A)$. Then $f(\partial A)$ is a subcomplex of $P^k$ and the quotient $\widehat{P}=P/f(\partial A)$ can be endowed with a simplicial structure of $k$-complex (the quotient map itself will not be simplicial, but we don't need that). 
    
    Now we have a map $\hat{f}:M_1\to \widehat{P}$ that behaves like $f$ on $M_1\setminus A$, and sends all $A$ to one point~$p_A$. To estimate the size of the fibers of this map, we need to consider a few cases. Let $x_1, x_2$ be two arbitrary points of a fiber $F_{\hat{p}}:=\hat{f}^{-1}(\hat{p})$ for some $\hat{p}\in\hat{P}$.
    \begin{itemize}
        \item If $\hat{p}\neq p_A:=f(\partial A)/f(\partial A)$, then $F_{\hat{p}}$ is a preimage of a single $p\in P$. The length of a minimizing geodesic $\gamma\subset M_1\#_AM_2$ from $x_1$ to $x_2$ is at most $W$ since they are in the same fiber $f^{-1}(p)$. If $\gamma$ does not intersect $M_2\setminus A$, then the copy of $\gamma$ in $M_1$ connects $x_1$ and $x_2$ by length $\leqslant W$. If $\gamma$ does intersect $M_2\setminus A$, not all of $\gamma$ is present in $M_1\setminus A$, and $d(x_1,x_2)$ there might differ from what it was in $M_1\#_AM_2$. In that case, let $y_1$ and $y_2$ be the first and last point of $\partial A$ in $\gamma$, so that the arcs $[x_i, y_i]$ of $\gamma$ lie entirely in $M_1\setminus \Int(A)$. Then
        $$d_{M_1}(x_1,x_2)\leqslant \bigl(d_{\gamma}(x_1,y_1)+d_{\gamma}(y_2,x_2)\bigr)+d_{M_1}(y_1,y_2)\leqslant$$
        $$\leqslant d_{\gamma}(x_1,x_2)+\diam_{M_1}(A)\leqslant W+\diam_{M_1}(A).$$
        
        \item If $\hat{p}= p_A$, then $f(x_i)\subset f(\partial A)$, so let $w_i\in\partial A$ be such that $f(w_i)=f(x_i)$. 
        
        Assume both $x_i$ are outside $A$. Let $\gamma_i$ be a minimizing geodesic in $M_1\#_AM_2$ from $x_i$ to $w_i$ and let $y_i$ be the first point of $\partial A$ in $\gamma_i$ so that $[x_i,y_i)\subset M_1\setminus A$. Then
        $$d_{M_1}(x_1,x_2)\leqslant  d_{\gamma}(x_1,y_1)+d_{M_1}(y_1,y_2)+d_{\gamma}(y_2,x_2)\leqslant$$
        $$\leqslant d_{\gamma}(x_1,w_1)+\diam_{M_1}(A)+d_{\gamma}(w_2,x_2) \leqslant 2W+\diam_{M_1}(A).$$
        If $x_i\notin A$, but $x_j\in A$, then
        $d_{M_1}(x_i,x_j)\leqslant  d_{\gamma}(x_i,w_i) +d_{M_1}(y_i,x_j)\leqslant W+\diam_{M_1}(A)$.
        
        If both $x_i$ are in $A$, $d_{M_1}(x_1,x_2)\leqslant \diam_{M_1}(A)$
    \end{itemize}

\end{proof}
\begin{proof}(of Theorem \ref{nocover}')
If $p_i:M_i\to P_i^k$ are maps with widths $\e$-close to $\UW_k(M_i)$, then we can define a map $p_\#$ from $M_1\#_A M_2$ to $P_1$ and $P_2$ glued along $p_i(\partial A)$. On $M_i\setminus \Int(A)$ we define $p_\#$ as the restriction of~$p_i$. Then the fibers of $p_\#$ are, at worst, fibers of $p_i$, joined along $\partial A$, so to get from one point of $p_\#^{-1}(x)$ to another we get at worst a chain (fiber of $p_i$)-($\partial A$)-(fiber of $p_j$). The more careful distance calculations are identical to those of theorem~\ref{nocover}.
\end{proof}

\section{Refined ``No Covers'' Case: contracting $\partial A$}
\label{contract_section}

The reason for the coefficient 2 in the theorem~\ref{nocover} was that when we build a map $M_1\to \hat{P}^k$ from a map $M_1\#_AM_2\to P^k$, we collapsed all $A$ to a single point, so all fibers that originally intersected $\partial A$ became a single larger fiber. This mixing of fibers can be avoided when extending the map to $A$, if instead of collapsing $A$ we use a null-homotopy of $\partial A$ in $P^k$. The following lemma then allows to extend the mapping to a collar neighborhood in a useful way.

\begin{lem}\label{if_contracts}
    Let $f:M\to P$ be a proper simplicial mapping to a complex $P^{k>0}$ and let $\Sigma$ be a compact collection of connected components of $\partial M$ such that $f|_{\Sigma}$ is contractible. Let $M_\Sigma$ be $M$ with a collar neighborhood $\Sigma\times [0,1]$ attached to $M$ (by $\Sigma\times \{1\}$). Then there exists a proper simplicial mapping $f_\Sigma:M_\Sigma\to Q^k$ for some $Q^k\supset P^k$ such that $f_\Sigma|_M=f$, $\Sigma\times\{0\}$ is a fiber $f_\Sigma^{-1}(p_0)$, and any fiber that intersects both $M$ and $\Sigma\times[0,1]$ goes through $\Sigma$: $f_\Sigma(\Sigma\times[0,1])\cap P^k=f(\Sigma)$.
\end{lem}

\begin{proof}
    Let $H: \Sigma\times[0,1] \to P^k$ be a null-homotopy of $f=:f_1$ to a constant mapping $f_0:\Sigma \to p_0\in P^k$. By taking simplicial approximation, $H$ can be made simplicial as well. We may also assume $f^{-1}(p_0)=\varnothing$ and $H^{-1}(p_0)=\Sigma\times\{0\}$ by having $p_0$ to be a new point, attached to $P^k$ by an edge. 
    
    We would like to define $f$ on $\Sigma\times[0,1]$ as $H$. However, if some original fiber $f^{-1}(p)$ was far away from $\Sigma$, but $H$ also had some nontrivial fiber $H^{-1}(p)\neq\varnothing$ over the same point, then we would have no control over the size of this fiber. To avoid this situation, we add to $P^k$ extra cells that $H$ could use exclusively. That is, let $P'$ be a copy of $P$ (with corresponding copies $f'$ and $H'$ of $f$ and $H$) and define
    $$Q^k:=P^k\bigcup H'(\Sigma\times[0,1])/\sim$$
    where the equivalence $\sim$ identifies $f(\Sigma)$ of the first copy and $H'(\Sigma\times \{1\})=f'(\Sigma)$ of the second. Then we can define $f_\Sigma:M\cup_\Sigma(\Sigma\times[0,1])\to Q^k$ as $f$ on $M$ and as $H'$ on $\Sigma\times [0,1]$. Now any fiber that is both in $M$ and in $\Sigma\times [0,1]$ has to go through $\Sigma$, since it has to be over a point $x$ that is both in $P^k$ and $H'(\Sigma\times [0,1])$, which by construction of $Q^k$ means $x\in f(\Sigma)$.
\end{proof}

To apply this construction to our case, we need to ensure contractibility on $\partial A$:

\begin{lem}\label{all-contraction}
    Let $\Sigma^{n-1}\subset \partial M^n$ be a compact union of some connected components $\Sigma_i$ of $\partial M$. If $f:M\to P^k$ is a proper simplicial map to a complex $P^k$, then $f|_\Sigma$ is contractible in some $Q^k\supset P^k$, provided that either
    \begin{itemize}
        \item $k=1$ and $\pi_1(\Sigma_i)\to 0\in \pi_1(M)$ for each $i$, or $H^1(\Sigma,\Z)\cong 0$, or
        \item $k=n-1$, $M$ is complete oriented with $\dim(M)\geqslant 3$, and $\partial M=\Sigma$ is compact and connected.
    \end{itemize}
\end{lem}

\begin{proof}
    If $k=1$ then $P^k$ is a graph. If $\pi_1(\Sigma_i)\to 0\in \pi_1(M)$ then by composition $f_*:\pi_1(\Sigma)\to\pi_1(P^1)$ is 0 and $f|_{\Sigma_i}$ can be contracted by induction over skeleta since all obstructions to that are 0, and then all $f(\Sigma_i)$ brought together, adding, if necessary, edges to make $P^1$ connected. Another condition $H^1(\Sigma,\Z)=0$ also prohibits non-contractible maps into connected graphs. It is enough to check that for a single component $\Sigma_i$. Assume there is such map $f:\Sigma_i\to P^1$. Then $\pi_1(P^1)$ is free, and so is any subgroup of it, in particular $I=f_*(\pi_1(\Sigma_i))$. If $f|_{\Sigma_i}$ is not contractible, then $I$ is a non-trivial group; being free, it admits a surjective map $g:I\to\Z$. Then the composition
    $$\pi_1(\Sigma_i)\stackrel{f_*}{\to} f_*(\pi_1(\Sigma_i))=I\stackrel{g}{\to}\Z$$
    is surjective, contradicting $\Hom(\pi_1(\Sigma_i),\Z)\cong\Hom(H_1(\Sigma_i,\Z),\Z)$ being zero due to universal coefficients theorem. This concludes the proof in the case $k=1$.

    Now assume instead that $k=n-1$ and $M$ is complete oriented with connected compact boundary $\partial M=\Sigma$. Let $f:M\to P^{n-1}$ be a proper map. By adding cells of dimension $\leqslant n-1$ to $P^{n-1}$ we may set $\pi_i(P^{n-1})\cong 0$ for $i<n-1$, meaning that $P^{n-1}$ becomes homotopy equivalent to a wedge of spheres $\bigvee_i S_i^{n-1}$. This ensures that $f$ is contractible on $\partial A$ according to the following lemma~\ref{n-1-contraction} (where we label $m=n-1$).       
\end{proof}

\begin{lem}\label{n-1-contraction}
    Let $M$ be a complete oriented manifold of dimension $m+1\geqslant 3$ with compact connected boundary $\partial M$, and $P^{m}$ is a simplicial complex, homotopy equivalent to $\bigvee_i S_i^{m}$. If $f:M\to P^{m}$ is a proper mapping, then $f|_{\partial M}$ is null-homotopic.
\end{lem}

\begin{proof}
    We first sketch the geometric idea of the argument. 
    
    If $M$ is a complete $m+1$-dimensional manifold with compact boundary, and $P=\bigvee_iS_i^m$, then smooth maps to $P$ can be visualized by the preimages of regular values $p_i\in S_i^m$, which will be $m$-codimensional submanifolds. In $M$ that will be a collection of arcs with both endpoints on $\partial M$ if the map is proper. If $M$ is oriented and $\partial M$ is connected, the endpoints of each arc make equal and opposite contributions to the ``degree'' of the mapping $f|_{\partial M}$, so the ``total degree'' has to be zero.
    
    Now we make this argument more rigorous. Since $P^m\sim \bigvee_i S_i^m$ we can pick cells $C_i\subset P$ such that the fundamental classes $[C_i]\in H^m(P,P\setminus \Int (C_i))$ form the basis of $H^m(P^m)$. The obstruction to contracting a map $f|_{\partial M}$ is zero for skeleta up to dimension $m-1$ since $\pi_k(\bigvee_iS_i^m)\cong0$ for $k<m$. In dimension $m$ the obstruction lives in 
    $$H^m(\partial M, \pi_m(\bigvee_iS_i^m))\cong\pi_m(\bigvee_iS_i^m)\cong H_m(\bigvee_iS_i^m)\cong\bigoplus_i\Z_i$$
    where in first equality we use that $\partial M$ is a closed, connected and orientable $m$-manifold, so that $H^m(\partial M,G)\cong G$, and in the second we use $m>1$ for Hurewicz isomorphism. Under this isomorphisms the obstruction becomes the image $f_*([\partial M])\in H_m(P)$ of the fundamental class $[\partial M]$, and it is enough to check that it pairs to 0 with each $[C_i]$. That is true since $[\partial M]$ is a boundary:
    $$f_*([\partial M])\frown [C_i]= \partial[M]\frown f^*([C_i])=[M]\frown d(f^*([C_i]))=[M]\frown f^*(d[C_i])=[M]\frown f^*(0)=0.$$
    This calculation holds for ordinary homology if $M$ is compact, but if it's not, then $[M]$ only exists in Borel-Moore homology, and in that case $f^*([C_i])$ would have to be present in cohomology with compact support for pairings to make sense. But that is ensured by $f$ being a proper map. Thus, all obstructions to contract $f|_{\partial M}$ are 0, meaning, it is contractible.
\end{proof}

\begin{proof}[Proof of theorem~\ref{nocover1}]

Consider a proper mapping $f:M_1\#_AM_2\to P^k$ of width $\W(f)$ arbitrarily close to $\UW_{k}(M_1\#_AM_2)$ (we may assume $f$ to be simplicial by lemma~\ref{simp}). By lemma~\ref{all-contraction} we can add simplicies to $P^{k}$ to make $f$ contractible on the boundary $\partial A$, and lemma~\ref{if_contracts} allows to add some more simplicies and extend $f$ from $M_1\setminus A$ to a collar neighborhood of $\partial A$ inside $A$ (and further as a constant map to the rest of $A$) to get $f_A:M_1\to Q^k$, so that $f_A(A)$ and $f_A(M_1\setminus A)$ intersect only by $f_A(\partial A)$.

We claim that the width $\W(f_A)$ is at most $\W(f)+\diam _{M_1}(A)$. Consider $q\in Q^k$. Let $x_1,x_2\in f_A^{-1}(q)$. If both $x_i$ are in $A$, then $d_{M_1}(x_1,x_2)\leqslant\diam_{M_1}(A)$. If both $x_i$ are in $M_1\setminus A$, then, using the same estimates as in the first point in the proof of theorem~\ref{nocover}, we get $\diam(f_A^{-1}(q))\leqslant \W(f)+\diam _{M_1}(A)$. Finally, let $x_1\in M_1\setminus A$ and $x_2\in \Int(A)$. Since images of these regions under $f_A$ intersect only by $f_A(\partial A)$, there is $w\in \partial A$ so that $f_A(w)=f_A(x_1)=f_A(x_2)$. Let $y$ be the first point of $\partial A$ on the minimizing geodesic $\gamma$ from $x_1$ to $w$ in $M_1\#_AM_2$ so that $[x_1,y)\subset M_1\setminus A$. Then
$$d_{M_1}(x_1,x_2)\leqslant d_{\gamma}(x_1,y)+d_{M_1}(y,x_2)\leqslant d_{\gamma}(x_1,w)+\diam_{M_1}(A)\leqslant\W(f)+\diam _{M_1}(A).$$
\end{proof}

\begin{proof}[Proof of Theorem \ref{nocover1}']
    We start with near-optimal maps $p_i:M_i\to P^k_i$ and either ensure $p_i|_{\partial A}$ is contractible by applying lemma~\ref{all-contraction} to $M_i\setminus \Int(A)$ or $A$ in cases a) and b), or, in case c) note that $p_i|_A$ is itself a contraction of $p_i|_{\partial A}$. Using lemma~\ref{if_contracts}, we can extend $p_i$ to a collar neighborhood of $\partial A$ so that the other boundary of the collar is a single fiber of $p_i$. A small difference in this case is that we want the collar to be now on the outside of $A$, but that can be achieved by a homeomorphism that pushes it outside, that can be taken to be arbitrarily close to identity on $M_i$. Now we can glue both $p_i$ into a single map 
$$p_\#:M_1\#_AM_2\to P_1\cup P_2/x_1\sim x_2,$$
whose fibers are either $p^{-1}_\#(x_i)=\partial A$ or reside in a single $M_i\setminus A$ and have their width bounded, similarly to the original, by $\W(p_i)+\diam(\partial A)$ (plus an arbitrarily small correction for pushing collars outside of $A$).
\end{proof}

\section{Covers}
\label{cover_section'}
To deal with the case of universal covers, we would like a version of theorem~\ref{nocover1} that handles multiple connected sums all at once:

\begin{lem}\label{multiple}
    Assume the setting of theorem~\ref{nocover1}, but replace a single connected sum with a manifold $N$ that is joined to multiple $N_i$, each through a corresponding $A_i$ (embedded disjointly in $N$, and each in the corresponding $N_i$) to get $N_\#:=N\#_{\bigcup A_i}\bigcup N_i$. If the distance function on $\partial A_i$ is in $N$ at most $C\geqslant 1$ times of what it is in $N_\#$, then under the same assumptions we get
    $$\UW_{k}(N)\leq C\UW_{k}(N_\#)+2\sup_{i}(\diam_{N}(A_i)).$$
\end{lem}

\begin{proof} The argument goes just like in theorem~\ref{nocover1}: given a proper simplicial mapping $f:N_\#\to P^k$ with width $W(f)$ arbitrarily close to $\UW_k(N_\#)$, the combination of lemmas~\ref{all-contraction} and~\ref{if_contracts} allows to extend $f|_{\partial A_i}$ to $A_i$ (first to a collar neighborhood of $\partial A_i$ inside $A_i$, then as a constant map to the rest of $A_i$). To apply those lemmas to a single $A_I$ at a time, we view $\partial A_I$ as a compact boundary of a complete manifold $N_\#\setminus (N_I\setminus A_I)$. We get $f_i:A_i\to Q_i^k\supset P^k$ such that $f_i(A_i)\cap P^k=f(\partial A_i)$. We can glue these extension together: let $Q^k$ be the union of $Q_i^k$ along $P^k$, so that we can define $f_N:N\to Q^k$ by setting $f_N=f_i$ on $A_i$ and $f_N=f$ on $N\setminus\bigcup_i A_i$.

To estimate the size of a fiber over $p\in P^k$, consider $x_1,x_2\in f_N^{-1}(p)$. Assume neither $x_i$ are in any of $A_j$. Let $\gamma_0$ be a minimizing geodesic in $N_\#$ from $x_1$ to $x_2$. If it ventures into $N_1,\cdots N_n$, we want to cut those pieces out. We modify $\gamma$ inductively: let $y_i$ and $z_i$ be the first and the last point of $N_i$ in $\gamma_{i-1}$ and set $\gamma_i$ to be $\gamma_{i-1}\setminus[y_i,z_i]$. Without loss of generality, $[y_i,z_i]$ was the longest such segment in $\gamma_{n-1}$ as compared to similarly defined segments for remaining $N_{i+1},\cdots N_n$, so that for each remaining $N_m$, their points in $\gamma_i$ cannot be on both side of the cut segment $[y_i,z_i]$. That means that no~$[y_i,z_i]$ intersects a previous cut, so all $[y_i,z_i]$ are disjoint in $\gamma_0$. Now we can replace the cut pieces with their substitutes in $N$: let $\beta_i$ be a minimizing geodesic in $N$ from $y_i$ to $z_i$. By gluing in $\beta_i$ into $\gamma_n$ in place of $[y_i,z_i]$ we get a path $\alpha$ in $N$ from $x_1$ to $x_2$. By definition of $C$, $|\beta_i|\leqslant C|[y_i,z_i]|$, therefore
$$d_N(x_1,x_2)\leqslant|\alpha|=|\gamma_n|+\sum_i |\beta_i|\leqslant |\gamma_n|+\sum_i C|[y_i,z_i]|\leqslant C(|\gamma_n|+\sum_i |[y_i,z_i]|)=C|\gamma_0|\leqslant C\W(f).$$
If $x_i$ fall inside some $A_j$, let $w_i\in \partial A_j\cap f_N^{-1}(p)$, let $\gamma_0$ go from $w_1$ to $w_2$ instead, so that we have 
$$d_N(x_1,x_2)\leqslant d(x_1,w_1)+|\gamma_0|+d(w_2,x_2)\leqslant C\W(f)+2\sup_i(\diam A_i).$$
Finally, a fiber over $q\in Q_i^k\setminus P^k$ is bound in size by $\diam A_i$.
\end{proof}

\begin{proof}[Proof of theorem~\ref{cover}]
    Here is the standard observation that reduces the theorem to the lemma~\ref{multiple} we just proved:
    \begin{lem}\label{universals}
        If $A$ and $\partial A$ are simply-connected, then (the total space of) the universal cover $pr_\#:\widetilde{M\#_AM'}\to M\#_AM'$ is a connected sum of (many copies of) universal covers $\tilde{M}_g$ and $\tilde{M}_h'$ (along their copies of $A$). Each copy of $\partial A_i$ separates in $\widetilde{M\#_AM'}$.
    \end{lem}
    Given lemma~\ref{universals}, let $N\supset A_i$ be one copy of $\tilde{M}$ that participates in $\widetilde{M\#_AM'}$ as a connected summand, that is connected via copies $A_i$ of $A$ to complete manifolds $N_i$, each encompassing everything on the other side of $\partial A_i$, so that $\widetilde{M\#_AM'}=N\#_{A_i}\bigcup N_i$. Then lemma~\ref{multiple} applies, since universal covers are orientable, so we conclude
    $$\UW_k(\tilde{M})=\UW_k(N)\leqslant
    C\UW_k(N\#_{A_i}\bigcup N_i)+2\sup_i\diam_N(A_i)=
    C\UW_k(\widetilde{M\#_AM'})+2\diam_{\tilde{M}}(A).$$
\end{proof}

\begin{rem}\label{sparse}
    Note that if $\UW_k(N_\#)$ is less than the distance between separate $\partial A_i$ in $N_\#$, then the minimizing geodesic $\gamma_0$ connecting two points of a fiber in $N_\#$ cannot visit two different $\partial A_i$, so when considering a given fiber, we may ignore all the $N_i$ except for the one that is potentially touched by $\gamma_0$ (say, $N'$), and argue as in theorem~\ref{cover}, but estimate the size of the fiber as in theorem~\ref{nocover1} in the context of the connected sum of $N'$ and the rest of $\tilde{M}$ along $\partial A_1$, and get
    $$\UW_k(\tilde{M})\leqslant \UW_k(\widetilde{M\#_AM'})+\diam_{\tilde{M}}(A).$$
\end{rem}

\begin{proof}[Proof of lemma~\ref{universals}]
    The cover $pr_\#$ is trivial over $\partial A$ since it is simply connected, so the boundary of the manifold $pr_\#^{-1}(M\setminus \Int A)$ consists of copies of $\partial A$. Note that each $\partial A_i$  separates in $\widetilde{M\#_A M'}$ since otherwise it would have a nontrivial loop that crosses $\partial A_i$ only in one direction. Let's call $N_0$ one connected component of $pr_\#^{-1}(M\setminus \Int A)$, and let $N$ be $N_0$ with every boundary component $\partial A_i$ filled in by $A_i\cong A$, so that the cover $pr_\#|_{N_0}:N_0\to M\setminus \Int A$ can be extended to $pr:N\to M$. To show that it is a universal cover, we need to show $\pi_1(N)\cong 0$.
    
    Recall that each $\partial A_i$ separates, so to each $\partial A_i$ we can assign ``the other side piece'' $N_i$ so that $\widetilde{M\#_A M'}\cong N\#_{A_i}\bigcup N_i$. Then by Van Kampen's theorem $\pi_1(\widetilde{M\#_AM'})$ is a free product over $\pi_1(\partial A_i)\cong 0$ of $\pi_1(N_0)$ and $\pi_1(N_i\setminus A_i)$. Since the result is 0, then $\pi_1(N_0)\cong 0$. Again, $\pi_1(N)$ is a free product over $\pi_1(\partial A_i)\cong 0$ of $\pi_1(N_0)$ and $\pi_1(A_i)\cong 0$, so $\pi_1(N)\cong 0$. Therefore, $N$ is isomorphic to the universal cover $\tilde{M}$.
\end{proof}

\begin{proof}[Proof of Theorem \ref{cover}']
The proof is mostly identical to the proof of theorem~\ref{cover}, with similar modifications of going from theorem~\ref{nocover1} to theorem~\ref{nocover1}'. We consider almost optimal projections $p_i:\widetilde{M_i}\to P^k_i$. They can be upgraded to be contractible on each copy $\partial A_i$ by lemma~\ref{all-contraction} applied to (orientable, as a universal cover) $\widetilde{M_i}\setminus A_i$ in cases a) and b), or $p_i|_{\partial A_i}$ is already contractible by $p_i|_{A_i}$ if $A_i$ is a ball. Using lemma~\ref{if_contracts}, $p_i$ can be further changed so that each copy of $\partial A$ is its own separate fiber, like in theorem~\ref{nocover1}', with the same caveats. These maps $p'_i$ can be glued along $\partial A_{i,j}$ in the domain and $x_{i,j}$ in the target to get $p:\widetilde{M_1\#M_2}\to P^k$. The estimates of the fiber sizes, again, are identical to those in theorem~\ref{cover} (since every $\partial A_{i,j}$ is a separate fiber, any other fiber is confined to a single $N$ piece like in theorem~\ref{cover}), giving
$$\UW_k(\widetilde{M_1\#_A M_2})\leqslant C'\max_{i}(\UW_k(\widetilde{M_i}))+2\diam_{M_1\# M_2}(\partial A).$$
\end{proof}

\section{Counterexamples}
\label{counter_section}
    Here we show examples that show that some of the constants in our theorems cannot be improved --- at least in general enough settings. 

\subsection{Constant 2 in theorem~\ref{nocover}}    
    First we tackle theorem~\ref{nocover} and its estimate 
    $$\UW_{k}(M_i)\leq {\bf 2}\UW_{k}(M_1 \#_A M_2)+\diam_{M_i}(A).$$
    We want to show the necessity of the constant {\bf 2} in this formula. Consider, for example, the following construction when $k=1$. Let $C$ be a cone of radius 1 over a circle of length $4\pi$ and $r:C\to [0,1]$ be the distance function to the vertex. Let's use notation like $C_{r\leqslant a}$ as a shorthand for $r^{-1}([0,a])\subset C$. Define $$A=C_{r\leqslant 1-\e} \text{,\hspace{.5cm}} M_1=C\text{,\hspace{.5cm}and\hspace{.5cm}} M_2=C_{r\leqslant 1-\e+\delta}$$
    Then $M_1\#_A M_2$ is (almost) a strip of thickness $\e+\delta$, and $\UW_1(M_1\#_A M_2)\leqslant \e+\delta$. However, we will check that
    \begin{lem}\label{conew}
        Urysohn 1-width of a unit cone $C$ over a circle of length $4\pi$ is $\UW_1(C)=2$. 
    \end{lem}
    Thus we have
    $$\UW_1(M_1)=2=2(1-\e)+2(\e+\delta)-2\delta=\diam(A)+{\bf 2}\UW(M_1\#_A M_2)-2\delta,$$
    which shows the necessity of the factor 2 by letting $\delta\to 0$. In this example the spaces are non-smooth manifolds with boundary, but they can be replaced by smooth closed $(n\geqslant3)$-manifolds $M_i^n$ that are the boundaries of regular $\delta'$-neighborhoods of the original $M_i^2$ (to which $M_i^2$ retracts by a $\delta'$-small displacement, so its $\UW_k$ is at least $\UW_k(C)-2\delta'$). However, that would make $A^3$ no longer be a disk. In fact, if $A$ is a disk of dimension $n\geqslant 3$, theorem~\ref{nocover1} applies and the factor 2 is not necessary.

    \begin{proof}[Proof of lemma~\ref{conew}]
        Assume $\UW_1(C)<2$ so that there is a map $f:C\to P^1$ of width $W<2$, that we assume to be simplicial by lemma~\ref{simp}. The plan is to use a fiber contraction type argument to retract (some  of) $C$ to $\partial C$ in a way that leads to a contradiction in homology.
        
        Denote $f_\partial:=f|_{\partial C}$. Consider a point $x\in \partial C$ of a non-empty fiber $f_\partial^{-1}(p)$; assign that $x$ to be the south pole of $\partial C$. Then every point of the northern semicircle has distance $2$ to $x$ and therefore doesn't intersect $f_\partial^{-1}(p)$. Therefore the preimage $f_{\partial}^{-1}(p)$ is contained in an open southern semicircle that we denote $S(p)$. Assume that these fiber-containing semicircles $S(p)$ can be taken to depend continuously on $p$, and the whole $P^1$ is in the image $f(\partial C)$, so that we have a fibration $F:\mathcal{S}\to P^1$ where $\mathcal{S}=\{(x,p)\subset\partial C\times P\text{ such that } x\in S(p)\}$. Then the pullback of $F$ to $C$ has contractible fibers $S(p)$, so it admits a section $s(x)$ whose $\partial C$-coordinate behaves as identity map on $\partial C$. Such section provides a retraction of $C$ onto $\partial C$, which is a contradiction.

        We assumed that all of $P^1$ is in the image of $f$. If that is not the case, we only construct the lifting over $X=f^{-1}(f(\partial C))$. The subset $X$ forms a chain in $C$ between the ``outer boundary'' $\partial C$ and the ``inner boundary'' $\partial (C\setminus X)$, but the section $s$ is trivial over the inner boundary when it comes to homology. Indeed, $f(\partial (C\setminus X))$ is in the subcomplex $f(\partial C)$ and (the closure of) its complement, meaning, it doesn't go inside the top cells of $P$. So the section $s(\partial (C\setminus X))$ can be taken to factor through $P^{(0)}$, and will vanish in $H_1(\partial C)$. Thus, the section $s(X)$ is still a null-homology of $[\partial C]\in H_1(\partial C)$ and a contradiction.

        We also assumed that the semicircles $S(p)$ can be chosen continuously in $p$. To do so, first choose a neighborhood $U(p)$ of any point $p\in f(\partial C)$ so that a constant $S(p)$ would work over~$U(p)$ (that is, contain the fibers $f_\partial^{-1}(q)$ for $q\in U(p)$). It is possible since the opposite (closed) semicircle $N(p)=\partial C\setminus S(p)$ is compact, so $f(N(p))$ is a closed set that misses $p$; thus, we can set $U(p):=\partial C\setminus N(p)$. Now, by using a partition of unity subordinate to this cover, interpolate between different semicircles $S(p)$ to make them depend continuously on $p$.  
    \end{proof}

    We provided a counterexample for relaxing the coefficient in theorem~\ref{nocover} for $\UW_k$ when $k=1$. If $k>2$, the exact same reasoning applies (including the proof of lemma~\ref{conew}), only instead of a double-sized circle, $\partial C$ should be a double-sized sphere of dimension $k$.

\subsection{Constant $C$ theorem~\ref{cover}, case $k=1$}
    Let us rewrite the estimate of theorem~\ref{cover} in the following way:
    \begin{equation}\label{c1}
        \UW_{1}(\widetilde{M_i})-2\diam_{\tilde{M_i}}(A)\leq {\bf C}\UW_{1}(\widetilde{M_1 \#_A M_2}).
    \end{equation}
    We want to show that the constant $\bf C$ here cannot be made uniform in the current setting (at least for more general metric spaces or for manifolds of dimension $n\geqslant 4$). The idea is to start with $M_1$ being (a 2-skeleton of) a $K(\Pi,1)$ for certain finite group $\Pi$ that would make $\UW_1(M_1)$ large, but then force $\widetilde{M_1\#_AM_2}$ to have small diameter (and hence $\UW_1$) by replacing every fundamental domain $A_i$ with a space of tiny diameter. Now we describe the details.

    Let $\Pi_n=(\Z/2\Z)^{3n}$ and $K(\Pi_n,1)$ be the product of $3n$ projective spaces $\R P^\infty$ with the standard metric and cell structure; thinking of $S^1$ as $U(1)$, we will treat the 1-cell of $\R P^\infty$ as $U(1)/\{\pm 1\}$. The 2-skeleton $K(\Pi_n,1)^{(2)}$ will be (a preliminary) $M_1$ and we pick $A$ to be the most of $M_1$, specifically, the product of balls of radius $(\pi/2-\e)$ around the base point (or rather, its intersection with the 2-skeleton). Let $M_2$ be a skewed doubling of $A$, $M_2=A\sqcup_{\partial A}A'$, where $A'$ has the same metric as $A$ near the boundary, but much smaller metric away from it, so that the whole $A'$ is within $\e$ distance of its center. This is all set up so that taking the connected sum $M_1\#_AM_2$ effectively shrinks most of $M_1$ from being $A$ down to $A'$. The same happens in the universal cover in each fundamental domain, forcing 
    \begin{equation}\label{diam}
        \diam(\widetilde{M_1\#_AM_2})\leqslant (12n+2\sqrt{2}+2)\e,
    \end{equation}
    since for an arbitrary point in $\widetilde{M_1\#_AM_2}$ it takes at most $\sqrt{2}\e$ to reach the nearest copy of $A'$, then $\e$ to get to its center, and $4\e$ to move from one copy $A'_g$ to a different one $A'_h$ if $g,h\in (\Z/2\Z)^{3n}$ differ by a single coordinate. 

    On the other hand, we can estimate $\UW_1(\tilde{M_1})$:
    \begin{lem}\label{rp2w}
        For the $M_1$ constructed above ($M_1:=K((\Z/2\Z)^{3n},1)^{(2)}$) we have 
        $$\UW_1(\tilde{M_1})\geqslant \frac{n-2\sqrt{2}}{8}\pi$$
    \end{lem}
    Together with $\diam(A)=\sqrt{2}(\pi-2\e)$ this shows that $\UW_1(\tilde{M_1})-2\diam(A)>1$ for large enough~$n$. And since by~(\ref{diam}) taking $\e\to 0$ forces $\diam(\widetilde{M_1\#_AM_2})\to 0$, the constant $\bf C$ in~(\ref{c1}) has to go to infinity and cannot be uniform.

    This example uses a 2-complex with piecewise Riemannian metric. One can turn it into a manifold by taking the boundary of its tiny normal neighborhood (the dimension needs to be at least 2 higher than the original in order to keep the fundamental group unchanged). The set $A$ can be kept a Riemannian ball (any compact manifold with boundary can be ``filled in'' by an embedded Riemannian ball that is $\e$-dense).

    \begin{proof}[Proof of lemma~\ref{rp2w}]
        The method for this proof is based on Gromov's fiber contraction argument and is one of the standard tools for estimating $\UW_k$ from below. We map our space $X$ onto a plane so that one of the contractible loops of $X$ encircles a big disk in the plane. If $X$ were almost 1-dimensional (had small $\UW_1$), then we could slightly homotope the mapping to factor through its 1-dimensional approximation, but that is impossible for a map stretched over a big disk for homological reasons. The details of this argument can be followed from Corollary~2.3 of~\cite{bb2021local} and Lemma~5.2 of~\cite{guth2005lipshitz}. We copy here Corollary~2.3 without proof and state it in a less general form for the sake of simplicity:
        \begin{lem}\label{fibercontract}
            Let $f:X\to \R^n$ be a 1-Lipschitz map from a metric space $X$, such that for some closed subset $X_0\subset X$ the induced map $f_* :H_n(X,X_0)\to H_n(\R^n,U_{2\rho}(f(X_0)))$ is non-trivial (here $U_r(S)$ denotes the neighborhood of $S$ of radius $r$). Then $\UW_{n-1}(X)\geqslant \rho$.
        \end{lem}
        To apply this lemma, we want some null-homologous curve $X_0$ to be stretched by some 1-Lipschitz map $f$ in such a way that $f(X_0)$ is not null-homologous in its $2\rho$-wide neighborhood. We set $X_0$ to be the following geodesic $\gamma(t)\subset \tilde{M_1}$ for $t\in \R/(3n\Z)$. We set
        $$\gamma(t):=(e^{i\pi x_1(t)},\cdots,e^{i\pi x_k(t)},\cdots)\subset \Bigl(\prod_{k=1}^{3n}U(1)\Bigr)^{(2)}\subset \tilde{M_1}$$
        where for integer $t$'s all $x_k(t)$ are 0 except for $k\in [t,t+n)$, for which $x_k=1$ (the indexes $k$ are also thought as belonging to $\Z/(3n\Z)$, meaning $x_{k+3n}\equiv x_k$). For non-integer $t$'s we define $x_k(t)$ via linear interpolation between nearest integer points.
        
        Now define $f:\tilde{M_1}\to\R^2$ as $f:x\mapsto \bigl(d(x,\gamma(0)),d(x,\gamma(n))\bigr)$. 
        \begin{lem}\label{geodesic}
            For $t,\tau\in \Z$ the distance $d(\gamma(t),\gamma(\tau))$ is equal to $\sqrt{2}\pi/2$ times the number of different coordinates in $\gamma(t)$ and $\gamma(\tau)$.
        \end{lem}
        Thus, we can describe $f(\gamma(t))$ for $t\in\Z$. If we denote by $\Delta(n)$ the closed broken line in $\R^2$ on the vertices $(0,n)$, $(n,0)$ and $(n,n)$, then $f(\gamma(t))$ follows the integral points of $\Delta(n)$, that are scaled up by $\sqrt{2}\pi$. Thus, the whole $f(\gamma)$ follows $\Delta':=\Delta(\sqrt{2}\pi n)$ deviating by at most $\pi/\sqrt{2}$ in either coordinate (that is half the length between two consecutive integral points of $\gamma$). It follows that $f(\gamma)$ represents a basis element of $H_1(U_r(\Delta'))$ for such $r$ that $f(\gamma)\subset U_r(\Delta')$ and $U_r(\Delta')\sim S^1$, which is ensured by, respectively, $r\geqslant\pi$ and $r\leqslant r_0= \sqrt{2}\pi n/4$. Thus any contraction of $\gamma$, viewed as a 2-chain, would therefore map to a basis element in $H_2(\R^2,U_r(\Delta'))$ for such $r$, or in $H_2(\R^2,U_r(f(\gamma)))$ for $r\leqslant r_0-\pi$. At this point we can apply lemma~\ref{fibercontract} to a mappng $f/\sqrt{2}$ (scaled down to make it 1-Lipschitz) and conclude that 
        $$\UW_{1}(\tilde{M_1})\geqslant \frac{r_0 -\pi}{2\sqrt{2}}=\frac{n-2\sqrt{2}}{8}\pi$$
    \end{proof}
    \begin{proof}[Proof of lemma~\ref{geodesic}]
        We may assume that $t\leqslant\tau\leqslant t+n$, with other cases following by symmetry. We define some useful coordinates on $\tilde{M_1}$: let $\chi_i:S_i^\infty\to \R$ be the distance function to $1\in U(1)\subset S_i^\infty$ on the $i$-th $S^\infty$ factor of the product $\prod_{i=0}^{3n}S_i^\infty$ whose 2-skeleton is $\tilde{M_1}$. The function 
        $$F_t^\tau=\sum_{i=t}^{i<\tau}\bigl(\chi_{i}-\chi_{i+n}\bigr)$$
        is $\sqrt{2}$-Lipschitz on the 2-skeleton $\tilde{M_1}$, since on the product of two 1-cells it is a function with gradient $(\pm 1,\pm 1)$ and it is 1-Lipschitz on 2-cells of $S_i^\infty$. Since $F_t^\tau(\gamma(t))=(\tau-t)\pi$ and $F_t^\tau(\gamma(\tau))=-(\tau-t)\pi$, then $$d(\gamma(t),\gamma(\tau))\leqslant 2(\tau-t)\pi/\sqrt{2}.$$
        On the other hand, $\gamma$ is a path that realizes this distance. 
    \end{proof}

\subsection{Constant $C$ theorem~\ref{cover}, case $k=n-1$}
    Like in the previous section, we want to show that the constant $\bf C$ in the estimate 
    \begin{equation}\label{cn-1}
        \UW_{n-1}(\widetilde{M_i})-2\diam_{\tilde{M_i}}(A)\leq {\bf C}\UW_{n-1}(\widetilde{M_1 \#_A M_2}).
    \end{equation}
    of theorem~\ref{cover} cannot be uniform. We construct a space $M_1$ whose universal cover has big $\UW_{n-1}$, fill most of $M_1$ with a ball $A$ of bounded size and collapse it (replace with small $A'$ by taking connected sum with $M_2=A\sqcup_{\partial A}A'$) so that the universal cover loses most of its $\UW_{n-1}$.

    Let $T^n$ be a torus $\R^n/\Z^n$ and $A\subset T^n$ be a Riemannian ball in its top cell that is $\e$-dense within that cell. Let $\gamma_i$ be loops embedded in $T^n\setminus A$ without intersecting each other so that their classes form the standard basis of $(N\Z)^n\subset\Z^n\cong \pi_1(T^n)$. We define $M_1$ to be a result of gluing a disk $D^2_i$ along each $\gamma_i$. To make it a manifold, the attachment should be done as a surgery replacing a tubular neighborhood $S^1\times D^{n-1}$ of $\gamma_i$ with $D^2\times S^{n-2}$. We want large enough metric on $D^2$ that we specify later, but we make the other factor small: $\diam(S^{n-2})<\e$; this sphere factor also forces $n\geqslant 4$ for the fundamental group $\pi_1(M_1)\cong (\Z/N\Z)^n$ to not be affected.

    Let us verify that these $M_1$ and $A$ have the desired metric properties. If $\e$ is small enough and the replacement $A'$ had diameter $<\e$, then $\widetilde{M_1\#_AM_2}$ is a bunch of handles $D^2\times S^{n-2}$ attached to an $N^n$-sheet cover of $T^n$, which is crumpled to have diameter $<1$ by shrinking interiors of copies of $A$ down to $A'$. Therefore, we can map $\widetilde{M_1\#_AM_2}$ to a wedge of spheres $\bigvee_i S^2_i$ by collapsing the sphere component of each handle $D^2\times S^{n-2}$ to make it a disk, and sending the disk boundary and the remaining torus part to the base point of the wedge. By construction, the width of this projection is $<1$, so $\UW_k(\widetilde{M_1\#_AM_2})<1$ for $k\geqslant 2$.

    On the other hand, we will show that 
    \begin{lem}\label{tw}
        For the $M_1$ constructed above as $T^n$ with handles $D^2\times S^{n-2}$ attached along loops that span $(N\Z)^n\subset \pi_1(T^n)$, we have 
        $$\UW_{n-1}(\tilde{M_1})\geqslant N/8$$
        if the metrics on $D^2$ is large enough.
    \end{lem}
    Together with $\diam(A)\leqslant\sqrt{n}$ this shows the unboundedness of $\bf C$ in~(\ref{cn-1}) by letting $N\to\infty$ for a given $n$.
    \begin{proof}[Proof of lemma~\ref{tw}]
        If in the construction of $M_1$ we glued in actual disks $D^2$ instead of doing surgeries, the original torus would be an actual subset of our space and $\tilde{M_1}$ would have an isometric copy of a cube $[0,N/2]^n$ (at least if $D^2$ has large enough metric to not introduce shortcuts that would affect distances within the torus). However, our construction poked some holes in $T^n$ in the manifold case, so we need a more flexible argument.

        Again, we use lemma~\ref{fibercontract} by mapping the cube $C:=(0,N/2)^n\subset T^n$ (or rather, what remains of it in $\tilde{M_1}$) over the same cube in $\R^n$. We define the mapping $f:\tilde{M_1}\to \R^n$ on the torus region as the product of 1-Lipschitz maps $\R/(N\Z)\to [0,N/2]$ that act as identities on the $[0,N/2]$ half and fold the second half back. For simplicity we will be extending the mapping $f$ to disks $D^2$, because for the actual $D^2\times S^{n-2}$ the construction is the same, just with more cluttered notation. For each copy $\tilde{\gamma_i}$ of $\gamma_i$ in the cover let $\gamma_{i,C}$ be the half of it that may cross the cube $C$ (that is, has $i$-th coordinate in $(0,N/2)$), and let $D_{i,C}$ be the half of the corresponding disk $\tilde{D^2_i}$ that borders $\gamma_{i,C}$. We extend the mapping to $D^2$ so that $f(\partial D_{i,C} \setminus \gamma_{i,C})$ lies outside the target cube $(0,N/2)^2$ (which is possible since $\partial \gamma_{i,C}$ is mapped to the boundary of the cube). The metric on $D^2_i$ should be picked large enough for this map to be 1-Lipschitz. Now we have constructed a subset $S:=C\cup \bigcup D_{i,C}$ that is mapped by $f$ over $[0,N/2]^2$ so that $f(\partial S)$ is outside $(0,N/2)^2$. That means, we can apply lemma~\ref{fibercontract} to the mapping $f:\tilde{M_1}\to \R^n$.
        Since $f(\partial S)\subset \R^n\setminus (0,N/2)^n$, we have a following composition for any $\rho$:
        $$H_n(S,\partial S)\subset H_n(\tilde{M_1},\partial S)\stackrel{f_*}{\to}H_n\bigl(\R^n,U_{2\rho}(f(\partial S))\bigr)\to H_n\bigl(\R^n,U_{2\rho}(\R^n \setminus (0,N/2)^n)\bigr).$$
        The last homology group is $\Z$ for $2\rho<N/4$, in which case the full composition is an isomorphism (most target points of $(0,N/2)^n$ are not hit by almost 2-dimensional $D^2\times S^{n-2}$ and are only mapped onto by the cube $C$ with degree 1). So the maps, including the middle $f_*$, are non-zero for any $\rho<N/8$, making the lemma applicable. Therefore $\UW_{n-1}(\tilde{M_1})\geqslant N/8$. 
    \end{proof}

\printbibliography

\end{document}